\documentclass[11pt,a4paper]{article}
\usepackage[utf8]{inputenc}
\usepackage{amsmath,amsthm,amsfonts}
\usepackage{esint}
\def\strokedint{\fint}

\newcommand\weakto{{\rightharpoonup}}
\numberwithin{equation}{section}
\newtheorem{theorem}{Theorem}[section]

\newtheorem{remark}[theorem]{Remark}

\newtheorem{lemma}[theorem]{Lemma}

\newcommand\R{\mathbb{R}}
\newcommand\N{\mathbb{N}}
\newcommand\rank{\mathrm{rank}}

\newcommand\qc{\mathrm{qc}}
\newcommand\pc{\mathrm{pc}}
\newcommand\nem{\mathrm{ne}}
\newcommand\Id{\mathrm{Id}}
\newcommand\cof{\mathop{\mathrm{cof}}}
\newcommand\eps{\varepsilon}
\newcommand\calL{\mathcal{L}}

\newcommand\Emitg{E}
\newcommand\Emitgrel{\Emitg^*}
\newcommand\Enog{{I}}
\newcommand\Enogrel{\Enog^*}
\def\xone{a_0}

\begin{document}
\begin{center}
{ \LARGE
On the theory of relaxation in nonlinear elasticity
with constraints on the determinant
 \\[5mm]}
{\today}\\[5mm]
Sergio Conti$^{1}$ and Georg Dolzmann$^{2}$\\[2mm]
{\em $^1$ Institut f\"ur Angewandte Mathematik,
Universit\"at Bonn\\ 53115 Bonn, Germany }\\
{\em $^{2}$ Fakult\"at f\"ur Mathematik, Universit\"at Regensburg,\\
 93040
  Regensburg, Germany}
\\[3mm]
\begin{minipage}[c]{0.8\textwidth}\small
 We consider vectorial variational problems in nonlinear elasticity of the
 form $\Enog[u]=\int W(Du)dx$, where 
$W$ is continuous on matrices with positive determinant
 and diverges  to infinity along sequences of matrices whose determinant is positive and 
tends to zero.
 We show that, under suitable growth assumptions, the 
 functional $\int W^\qc(Du)dx$ 
is an upper bound on the relaxation of
$\Enog$, and coincides with the relaxation if the quasiconvex envelope $W^\qc$ of
$W$ is polyconvex and has $p$-growth from below with $p\ge n$. This includes several
physically 
relevant examples. We also show how a constraint of incompressibility can be
incorporated in our results.
\end{minipage}
\end{center}
\section{Introduction}

Starting with the work of Morrey \cite{Morrey1952}, the concept of
quasiconvexity has been 
fundamental in the study of the relaxation of vectorial problems in the
calculus of variations, 
see for example \cite{Dacorogna1989,MuellerLectureNotes,RoubicekBook1997}. In particular, 
 if $W:\R^{n\times 
  m}\to\R$ is continuous and has $p$-growth then the relaxation of the
functional $\Enog: W^{1,p}(\Omega;\R^m)\to\R$,
\begin{equation}\label{eqintW}
  \Enog[u]=\int_\Omega W(Du) dx
\end{equation}
is given by
\begin{equation}\label{eqintWqc}
  \Enogrel[u]=\int_\Omega W^\qc(Du) dx\,,
\end{equation}
where $W^\qc$ is the quasiconvex envelope of $W$, see
 \cite{Morrey1952,AcerbiFusco84,Dacorogna1989}. Here $\Omega\subset\R^n$ is a bounded Lipschitz set,
$W^\qc$ is defined by~(\ref{eqdefwqcdetp}) below.
The computation of $W^\qc$ is in general difficult, but it was performed in a
number of  special cases  with high symmetry, see for example
\cite{DeSimoneDolzmannARMA2002,ContiTheil2005,Silhavy2007,ContiDolzmanntwowell}. 
The key strategy is to construct specific test functions  using
lamination and rank-one convexity to prove an upper bound, and then to show
that the resulting expression is polyconvex, which delivers the lower bound.

One of the main applications of the vectorial calculus of variations is
nonlinear elasticity with $m=n$. The physical constraint of non-interpenetration of
matter leads naturally to the requirement of injectivity of the
deformation $u$, a complex nonlocal condition which is often replaced by the
simpler condition that $\det Du>0$ almost everywhere. Correspondingly, one 
assumes that the energy density $W$ diverges when the determinant of the
argument is positive and tends to zero. Such energy 
densities are not continuous on $\R^{n\times n}$ and do not have $p$-growth
from above for any $p$, hence the general relaxation theorem is not
applicable. 
Starting with the work of Ball
 \cite{Ball1977} a large body of work developed with the aim of proving
 existence of minimizers for variational problems with the constraint $\det
 Du>0$, mainly building upon the concept of polyconvexity. In contrast,
to the best of our knowledge, there is no physically-relevant functional
incorporating the nonlinear constraint $\det Du>0$ for which a nontrivial
relaxation is known. 

The significance of the constraint $\det Du>0$ depends dramatically on the
growth 
exponent $p$ of the energy, and two main regimes emerge. 
If $p\ge n$ the deformation $u$  is necessarily continuous. This follows from
the Sobolev embedding theorem  for
$p>n$ and in the case $p=n$ from the work of
Vodop$'$janov and Gol$'$d{\v{s}}te{\u\i}n 
\cite{VodopjanovGoldstein1976}, see also \cite{Sverak1988}. Further, if a
sequence $u_j$ converges weakly to some $u$ in $W^{1,n}$ and $\det Du_j>0$
almost everywhere, then necessarily $\det Du\ge 0$ almost everywhere. 
For $p>n$ this follows directly from the properties of the determinant, 
which in particular give  $\det Du_j\weakto \det Du$ in $L^1$
\cite[Cor.  6.2.2]{Ball1977}, and still
holds for $p=n$, see \cite[Th. 4.1(ii)]{BallMurat1984}. A related treatment
with the distributional determinant instead of the pointwise determinant is
still possible if $p>n-1$ and a generalized invertibility condition is used
instead of $\det Du>0$ \cite{MuellerSpector1995},
see also \cite{MuellerSivaloganathanSpector1999,ContiDeLellis2003,FonsecaLeoniMaly2005,HenaoMoracorral2010} for subsequent developments. 

The case $p<n$ is substantially different, since the deformations are
not continuous and can develop holes, as was first shown by Ball
\cite{Ball1982}. 
Correspondingly, the constraint of having positive determinant
does not pass to the limit and the relaxed problem has a substantially
different structure, see for example
\cite{BallMurat1984,KoumatosRindlerWiedemann1,KoumatosRindlerWiedemann2}
for further developments.  Relaxation in a related situation in 
which the constraints are lost after rank-one convexification was 
discussed in \cite{Benbelgacem2000}. We shall not discuss these  cases 
further here. 

The proof of the classical relaxation theorem for continuous integrands is
based on a truncation procedure in 
which one replaces  a sequence $u_j\in W^{1,p}$ which converges weakly to an
affine function by a sequence with the same weak limit, the same energy, and
which is affine on the boundary. It is currently unknown if a similar
construction can be done if a constraint on the determinant has to be
preserved. In two spatial dimensions and under the
assumption that $u$ is bilipschitz 
a solution was given in 
\cite{BenesovaKruzik}
 building upon an involved 
 construction of bilipschitz extensions by Daneri and Pratelli 
\cite{DaneriPratelli1}. 
The approximation of Hölder-continuous homeomorphisms was obtained in
\cite{BellidoMoracorral2011}. The situation with Sobolev  functions is
substantially more complex, and was up to now only solved in the case
 $p=n=2$, see \cite{IwaniecKovalevOnninen2011}. 
A related problematic arises in the relaxation of problems with mixed
growth, see for example  \cite{Kristensenlsc}.

In this paper we prove that the functional (\ref{eqintWqc}) gives an upper
bound on the relaxation of (\ref{eqintW}) for a class of energy densities
which are infinite on matrices $F$ with $\det F\le 0$ and have $p$-growth for
some $p\ge 1$ on the set $\{\det F>0\}$, see Theorem \ref{theoorientation} below.
If $W^\qc$
is polyconvex and $p\ge n$ then
(\ref{eqintWqc}) coincides with the relaxation of  
 (\ref{eqintW}), see Theorem \ref{theocorollaryorientat} below.  The 
growth assumption can be somewhat relaxed if a suitable integrability of the cofactor is assumed, see Remark \ref{remarkgener}.  Since all known explicit quasiconvex
 envelopes  $W^\qc$ are 
 quasiconvex, our result fully characterizes the relaxation  in all cases
 where 
 $W^\qc$ has been computed and the pointwise determinant constraint survives the
 relaxation. 
We also show that our results can be generalized to problems where $\det Du=1$
almost everywhere, see Section \ref{subsecapplicincompr} below.

Notation: We denote 
by  $\R^{n\times n}_+=\{F\in \R^{n\times n}: \det F>0\}$ the set of orientation-preserving matrices,   by $B(r,x_0)$ the open ball in $\R^n$ and set
$B_r=B(r,0)$. Finally $\strokedint_E f\, dx$ denotes the mean value of $f$ on $E$.

We define 
the quasiconvex envelope $W^\qc:\R^{n\times n}\to[0,\infty]$ of a Borel-measurable function 
$W:\R^{n\times n}\to[0,\infty]$ by 
\begin{equation}\label{eqdefwqcdetp}
  W^\qc(F)=\inf\{\strokedint_{B_1} W(D\varphi)\, dx: \varphi\in
  W^{1,\infty}(B_1;\R^n), \varphi(x)=Fx \text{ for } x\in \partial B_1 \}\,.
\end{equation}
This is not necessarily the same as the largest finite-valued quasiconvex function below $W$, see Remark \ref{remarkwqc} below.

We say that a function $f:\R^{n\times n}\to [0,\infty]$ is polyconvex if
there is a lower semicontinuous and convex function $g:\R^{\tau(n)}\to[0,\infty]$ such
that $f(F)=g(M(F))$, where $M(F)$ denotes all minors of $F$
 \cite{Ball1977,Dacorogna1989,MuellerLectureNotes}. 
In particular, if $n=2$ then
$M(F)=(F,\det F)$, if $n=3$ then $M(F)=(F,\cof F, \det F)$.
The requirement of lower semicontinuity of $g$ is often not included 
in the definition but instead enforced through appropriate growth conditions. 
For  finite-valued functions this makes no difference, but for 
extended-valued we believe the present one to be the definition more naturally
related to lower semicontinuity, as the example 
$g(\det F)=0$ if $\det F>0$, $g(\det F)=\infty$ otherwise, with the 
sequence $u(x)=x/j$ shows, see also the discussion in  \cite{Mielke2005}.

\section{Main results}

\subsection{Relaxation of orientation-preserving models}
\label{secapplicaorientpres}
The main result of the paper is a relaxation theorem
for coercive variational problems 
in nonlinear elasticity incorporating a constraint on the determinant, see Theorem \ref{theocorollaryorientat} below. 
Our key new contribution is a construction for the upper bound which preserves the positive-determinant
constraint and leads to the following statement.
\begin{theorem}\label{theoorientation}
Let $W\in C^0(\R^{n\times n}_+,[0,\infty))$ obey
\begin{equation}\label{eqgrowthwpd}
  \frac1c |F|^p +\frac1c\theta(\det F)-c\le W(F)\le c|F|^p+c\theta(\det F)+c
\end{equation}
for some $p\ge 1$ and $c>0$,  where 
$\theta:(0,\infty)\to[0,\infty)$ is convex and satisfies
\begin{equation}\label{eqasstheta}
  \theta(xy)\le c(1+\theta(x))(1+\theta(y))  \text{ for all
    $x,y\in(0,\infty)$}, 
\end{equation}
and extend $W$ to $\R^{n\times n}$ by $W(F)=\infty$ if $\det F\le 0$.
Let $W^\qc$ be defined
as in (\ref{eqdefwqcdetp}), $\Omega\subset\R^n$ open, bounded and 
Lipschitz. For any $u\in W^{1,p}(\Omega;\R^n)$ there is a sequence $u_j\in
W^{1,p}(\Omega;\R^n)$
which converges weakly to $u$ such that $u_j-u\in W^{1,p}_0$ for all $j$ and 
\begin{equation*}
  \limsup_{j\to\infty} \int_\Omega W(Du_j)dx\le \int_\Omega W^\qc(Du)dx\,.
\end{equation*}
\end{theorem}
\begin{proof}
If $\det Du>0$ almost everywhere the statement follows from Lemma \ref{lemmaconstr2} in Section \ref{secorientationpres} below. 
From the definition one immediately obtains $W^\qc(F)=\infty$ if $\det F\le 0$, therefore in the 
other case a constant sequence will do.
\end{proof}
If the coercivity exponent $p$ is at least $n$,
then the determinant is an $L^1$ function and weakly continuous in compact
subsets \cite{Mueller1990}, therefore the constraint on the 
determinant passes to the limit.
Complementing  Theorem~\ref{theoorientation}
with existing compactness and lower semicontinuity  results, 
based on the concept of   polyconvexity \cite{Ball1977},
leads to a full relaxation and existence statement.

\begin{theorem}\label{theocorollaryorientat}
Let $W\in C^0(\R^{n\times n}_+,[0,\infty))$ obey
(\ref{eqgrowthwpd}--\ref{eqasstheta})
with   $p\ge n$ and
\begin{equation*}
\lim_{t\to0}\theta(t)=\infty\,,
\end{equation*}
and extend $W$ by $W(F)=\infty$ to the set  $\{\det F\le 0\}$.
Let  $\Omega\subset\R^n$ be an open, bounded,
Lipschitz, connected set, 
\begin{align*}
 X=\{u\in W^{1,p}(\Omega;\R^n): \det Du>0 \text{ a.e. }\}\,,
\end{align*}
and $f\in C^0(\R^n)$ with $|f(t)|\le c (1+|t|^q)$ for some $q\in[0,p)$.
We define $W^\qc$ as in (\ref{eqdefwqcdetp})
and the functionals 
$\Emitg, \Emitgrel:L^1(\Omega;\R^n)\to \R\cup\{\infty\}$ by
\begin{equation*}
\Emitg[u]=\int_\Omega \left(W(Du) + f(u)\right) dx \text{ and }
\Emitgrel[u]=\int_\Omega \left(W^\qc(Du) + f(u)\right) dx
\end{equation*}
for $u\in X$, and $\Emitg=\Emitgrel=\infty$ on $L^1\setminus X$.
Finally assume that $W^\qc = W^\pc$. 
Then the following assertions hold:
\begin{enumerate}
\item 
$\Emitgrel$ is the relaxation of $\Emitg$ with respect to strong $L^1$ convergence, in the sense that
\begin{equation*}
\Emitgrel[u]=\inf\{\liminf_{j\to\infty} \Emitg[u_j]: u_j \in L^1(\Omega;\R^n), u_j\to u \text{ in } L^1\}\,.
\end{equation*}
\item The same holds if,
for any given relatively open set 
$\Gamma_D\subset\partial\Omega$  and $u_0\in X$,
the functionals $\Emitg$ and $\Emitgrel$ are set to be $\infty$ outside
  \begin{equation*}
\widetilde X=X\cap \{u=u_0\text{ on }\Gamma_D\}\,.
  \end{equation*}
\item The functional
$\Emitgrel$ has a minimizer in the space $\widetilde X$. 
\end{enumerate}
\end{theorem}
\begin{proof}
  By Theorem \ref{theoorientation} for any $u\in X$ there is $(u_j)_{j\in \N}\subset X$ with $u_j=u$ on $\partial\Omega$, $u_j\to u$ for $j\to\infty$ in $L^p$,
  and $\limsup_{j\to\infty} \Emitg[u_j]\le \Emitgrel[u]$. This proves the upper bound in both cases.

  Let now $(u_j)_{j\in \N}$ be a sequence in $X$ with $\Emitgrel[u_j]\le C<\infty$ for all
  $j$. From~(\ref{eqgrowthwpd}) one immediately obtains $W^\qc(F)\ge |F|^p/c-c$.
  The growth condition on $f$ ensures, since $q<p$, that $\int_\Omega W^\qc(Du)dx\le
  C'<\infty$ for all $j$.
  Taking a subsequence we can assume $u_j\weakto u$ in $W^{1,p}$ for some
  $u\in W^{1,p}(\Omega;\R^n)$.
  By the continuity of the trace, if $u_j=u_0$ on $\Gamma_D$ for all $j$ then $u=u_0$ on $\Gamma_D$.

  In order to show that  $u\in X$ it only remains to prove the condition on the determinant. 
By \cite{Mueller1990} we have $\det Du_j\weakto \det Du$ in $L^1(K)$ for all $K\subset\subset\Omega$. Since $\theta$ diverges at 0  the  weak limit is positive almost everywhere and  $u\in X$.

     To prove lower semicontinuity we
     let $W^\qc(F)=g(M(F))$, with $g$ convex and lower semicontinuous, and 
     fix a compact set  $K\subset\Omega$. As discussed above,
     we have $\det Du^j\weakto\det Du$ in 
     $L^1(K)$ for all $K\subset\subset\Omega$. The other minors converge also weakly in $L^1$, since
     $u_j\weakto u$ in $W^{1,p}$ with $p>n-1$. Therefore $M(Du^j)\weakto M(Du)$ in $L^1(K)$ and using
     Jensen's inequality and the
     convexity of $g$ we obtain
     \begin{alignat*}1
       \int_K W^\qc(Du)\, dx&=\int_K g(M(Du))\, dx\le\liminf_{j\to\infty} \int_K g(M(Du_j))\, dx 
\\&       \le \liminf_{j\to\infty} \int_\Omega W^\qc(Du_j)\, dx\,.
     \end{alignat*}
     The term $\int_\Omega f(u) dx$ is continuous.
     Taking the supremum over all compact subsets $K\subset\Omega$  gives $\Emitgrel[u]\le \liminf \Emitgrel[u_j]\le \liminf \Emitg[u_j]$ and
  concludes the proof.
\end{proof}

\begin{remark}\label{remarkgener}
  \begin{enumerate}
\item  The result of Theorem \ref{theoorientation} can be extended  to functions  $W$  which obey 
\begin{equation}\label{eqwprod1}
    W(FG)\le c (1+W(F))(1+W(G))
\end{equation}
and
\begin{equation}\label{eqwprod2}
  \frac1c |F|-c\le W(F)\le c W^\qc(F)+c
\end{equation}
instead of (\ref{eqgrowthwpd}) and (\ref{eqasstheta}). We discuss  in Appendix \ref{sectmultipl}  the required modifications to the proof.
  \item Using this generalization, one can extend Theorem \ref{theocorollaryorientat} to a situation in which $W$ obeys  (\ref{eqwprod1}), (\ref{eqwprod2}) and the growth condition
    \begin{equation*}
        \frac1c |F|^p +\frac1c|\cof F|^q+\frac1c\theta(\det F)-c\le W(F)\,,
    \end{equation*}
 corresponding to the spaces $\mathcal{A}_{p,q}$ introduced by  Ball \cite{Ball1977}, see also \cite{Sverak1988}. 
\item A different picture arises if one instead uses a constraint on
the pointwise determinant, with the material becoming substantially softer,
see \cite{Ball1982,BallMurat1984} and \cite{KoumatosRindlerWiedemann1,KoumatosRindlerWiedemann2,KoumatosRindlerWiedemann3}.
  \end{enumerate}
\end{remark}

\subsection{Relaxation of incompressible models}
\label{subsecapplicincompr}
We deal with integrands which are defined on the set of volume-preserving
matrices
 $\Sigma=\{F\in \R^{n\times n}: \det F=1\}$
and which have $p$-growth. In this framework, we prove the following result. 
\begin{theorem}\label{theovolume}
Let $W\in C^0(\Sigma,[0,\infty))$ obey
\begin{equation}\label{eqgrowthwincompr}
  \frac1c |F|^p-c\le W(F)\le c|F|^p+c
\end{equation}
for some $p\ge 1$ and $c>0$, and set $W(F)=\infty$ if $\det F\ne 1$. Let $W^\qc$ be defined as in (\ref{eqdefwqcdetp}),  $\Omega\subset\R^n$ open, bounded and 
Lipschitz. For any $u\in W^{1,p}(\Omega;\R^n)$ there is a sequence $u_j\in
W^{1,p}(\Omega;\R^n)$
which converges weakly to $u$ such that $u_j-u\in W^{1,p}_0$ and 
\begin{equation*}
  \limsup_{j\to\infty} \int_\Omega W(Du_j)dx\le \int_\Omega W^\qc(Du_j)dx\,.
\end{equation*}
\end{theorem}
\begin{proof}
     The  statement follows from Lemma \ref{lemmarecoveryincompr} below.  
\end{proof}
Also in this case, if  coercivity is sufficient  a full relaxation statement follows.
\begin{theorem}
Let $W\in C^0(\Sigma,[0,\infty))$ obey
(\ref{eqgrowthwincompr}) for some  $p\ge n$, let $\Omega\subset\R^n$ be open bounded,
Lipschitz, connected, 
\begin{align*}
 X=\{u\in W^{1,p}(\Omega;\R^n): \det Du=1 \text{ a.e.}\}\,.
\end{align*}
We  set $W(F)=\infty$ if $\det F\ne 1$, define $W^\qc$ as in (\ref{eqdefwqcdetp})
and, for $f\in C^0(\R^n)$ with $|f(t)|\le c (1+|t|^q)$ for some $q<p$, 
\begin{equation*}
\Emitg[u]=\int_\Omega \left(W(Du) + f(u)\right) dx \text{ and }
\Emitgrel[u]=\int_\Omega \left(W^\qc(Du) + f(u)\right) dx
\end{equation*}
for $u\in X$, and $\Emitg=\Emitgrel=\infty$ on $L^1\setminus X$.
Finally assume that $W^\qc=W^\pc$. 
Then the following assertions hold.
\begin{enumerate}
\item 
$\Emitgrel$ is the relaxation of $\Emitg$ with respect to strong $L^1$ convergence, in the sense that
\begin{equation*}
\Emitgrel[u]=\inf\{\liminf_{j\to\infty} \Emitg[u_j]: u_j \in L^1(\Omega;\R^n), u_j\to u \text{ in } L^1\}\,.
\end{equation*}
\item The same holds if,
for any given relatively open set 
$\Gamma_D\subset\partial\Omega$  and $u_0\in X$,
the functionals $\Emitg$ and $\Emitgrel$ are set to be $\infty$ outside
  \begin{equation*}
\widetilde X=X\cap \{u=u_0\text{ on }\Gamma_D\}\,.
  \end{equation*}
\item The functional
$\Emitgrel$ has a minimizer in the space $\widetilde X$. 
\end{enumerate}
\end{theorem}
\begin{proof}
The proof is analogous to the proof of Theorem \ref{theocorollaryorientat}. 
\end{proof}

\subsection{Examples}
We first consider the two-well problem in two dimensions, a classical model for
 microstructure in martensite \cite{BallJames87,BallJames92,Bhatta,Ball2002}.
Precisely, we define  
\begin{equation*}
W_{2W}(F)=
  \mathrm{dist}^2(F, SO(2) U_1 \cup SO(2) U_2 ) + \theta(\det F)\,,
\end{equation*}
where
\begin{align*}
   U_1=
  \begin{pmatrix}
    \lambda&0\\0&1/\lambda
  \end{pmatrix}\,,\quad
  U_2=
  \begin{pmatrix}
    1/\lambda&0\\0&\lambda
  \end{pmatrix}\,,
\end{align*}
for some fixed $\lambda>1$ are the eigenstrains of the two martensitic phases
and  $\theta\in C^0((0,\infty),[0,\infty))$ is a convex function which obeys
(\ref{eqasstheta}) and with  $\lim_{t\to0} \theta(t)=\infty$, for example 
$f(t)=(t-1/t)^2$, extended with $\theta=\infty$ on $(-\infty,0]$. One is then interested in the functional  
$\Enog_{2W}: W^{1,2}(\Omega;\R^2)\to[0,\infty]$, 
\begin{equation*}
\Enog_{2W}[u]=\int_\Omega W_{2W} (Du)\, dx\,.
\end{equation*}
In \cite{ContiDolzmanntwowell} it was shown that the quasiconvex envelope 
of $W_{2W}$ is given by 
  \begin{equation*}
    W_{2W}^\qc(F) = 
h(|Fv|, |Fw|, \det(F)) + \theta(\det F)\,,
  \end{equation*}
where $v=(e_1+e_1)/\sqrt2$, 
 $w=(e_1-e_1)/\sqrt2$ and
$h$ is defined by
  \begin{equation*}
  h(x,y,d)= \min_{\xi\in [x,\infty),\,\eta\in [y,\infty)} \left(
  \xi^2+\eta^2+|U_1|^2 -2 \sqrt{A(\xi,\eta,d)} \right)\,,  
  \end{equation*}
with
\begin{alignat*}1
  A(x,y,d)& =(x^2+y^2)\frac{|U_1|^2}2 +
    (\lambda^2-\frac1{\lambda^2}) \sqrt{x^2y^2-d^2} + 2 d \,,
\end{alignat*}
and that $W_{2W}^\qc$ is polyconvex. Theorem \ref{theocorollaryorientat}
then shows that the relaxation of $\Enog_{2W}$ is given by
\begin{equation*}
\Enogrel_{2W}[u]=\int_\Omega W^\qc_{2W} (Du)\, dx\,.
\end{equation*}

A corresponding result holds for models with one potential well only, which can be recovered setting $\lambda=1$ in the previous expressions, the  quasiconvex envelope is given for example in \cite{Silhavy1999}. 

A related situation with the incompressibility constraint can be obtained from
the study of nematic elastomers \cite{DesimoneDolzmann00,Silhavy1999,DeSimoneDolzmannARMA2002,WarnerTerentjevBook,Silhavy2007}. They are composite materials in which a rubber (polymer) matrix is coupled to a nematic liquid crystal. The rubber has entropic elasticity and is usually modeled as incompressible; the ordering of the nematic liquid crystal leads to elongation in the direction of the nematic order parameter. After minimizing out the nematic director, the standard model \cite{WarnerTerentjevBook}  can be cast in the form
\begin{equation*}
  W_\nem(F)=
  \begin{cases}
    \sum_{i=1}^n \left(\frac{\lambda_i(F)}{\gamma_i}\right)^p
& \text{ if $\det F=1$}\\
    \infty & \text{ if } \det F\ne 1
  \end{cases}
\end{equation*}
where $\lambda_i(F)$ are the singular values of $F$, i.e., the eigenvalues of $(F^TF)^{1/2}$, and $\gamma_i\in (0,\infty)$ are material parameters with $\prod_{i=1}^n \gamma_i=1$. 

In two dimensions one can assume  $\gamma_2=1/\gamma_1>1$ and, taking the natural exponent $p=2$, one can show that \cite{DesimoneDolzmann00}
\begin{equation*}
 W_\nem^\qc(F)=
 \begin{cases}
   \infty & \text{ if } \det F\ne 1\\
2 & \text{ if } \lambda_2(F)\le \gamma_2 \text{ and } \det F=1\\
   W_\nem(F) & \text{ otherwise.}
 \end{cases}
\end{equation*}
Further, the function $W_\nem^\qc$ is polyconvex  \cite{DesimoneDolzmann00}. 
Therefore Theorem \ref{theovolume} shows that the relaxation of $\Enog_\nem$ is given by $\Enogrel_\nem$. 

In the physically relevant situation $n=3$ a similar expression is 
analytically known and turns out to be polyconvex for all $p\ge 1$ 
\cite{DeSimoneDolzmannARMA2002}. Theorem \ref{theocorollaryorientat} then  states that 
 $\Enogrel_\nem$ is the relaxation of  $\Enog_\nem$ for all $p\ge 3$. This does not include, 
however, the physically relevant exponent $p=2$. In this case, as noted in the 
introduction, the variational model with the pointwise constraint on the determinant would
predict cavitation  under tension.
Tensile experiments are however carried out in a regime 
in which cavitation 
does not occur, possibly due to metastability or to additional energy 
contributions not modeled by $W_\nem$ \cite{WarnerTerentjevBook}. Indeed, 
numerical simulations based on the quasiconvex envelope $W_\nem^\qc$ and 
excluding cavitation via the choice of a the finite-element space of 
continuous function lead to good agreement with experimental observations 
\cite{ContiDesimoneDolzmann2002}.

\subsection{Strategy of the proof}
The ``classical'' construction of the upper bound for relaxation theorems is
based on density and interpolation. Given a function $u\in W^{1,p}$
one first approximates it strongly in $W^{1,p}$ by a piecewise affine
function and then uses a ``good'' test function $\varphi$ in each of the
pieces where $u$ is affine. In doing this it is important that the energy is
continuous along the approximating sequence. In our
case the energy is infinite on all deformations which are not orientation-preserving, therefore one would need  to approximate orientation-preserving 
$W^{1,p}$ maps with piecewise affine orientation-preserving maps, 
a problem which, as we discussed 
above, is very difficult and in general unsolved.

Here we show that  the passage through piecewise affine maps is not needed.
The key idea is to locally use the {\em composition} of the limiting map $u$
with the map with oscillating gradient $\varphi$. Taking the composition  preserves the
conditions on the determinant, and the map $u\circ\varphi$ belongs to $W^{1,p}$ 
since $u\in
W^{1,p}$ and $\varphi$ is Lipschitz. The key point is to prove that the unrelaxed energy 
of the composition is close to the relaxed energy of $u$.

For simplicity we focus here on a small ball $B$ where $Du$ is close
to the 
identity (after a change of variables this is true around any Lebesgue point
of $Du$). Let $\varphi\in W^{1,\infty}(B,\R^n)$ be a map from 
(\ref{eqdefwqcdetp}), which in particular is the identity on the boundary of $B$. The key idea 
is to define $z=u\circ \varphi$, so that $Dz=Du\circ\varphi D\varphi$. If $Du$ is close to the identity and $Dz$
is close to $D\varphi$ one then (heuristically) obtains
\begin{equation*}
\int_B W(Dz)\, dx\sim \int_B W(D\varphi)\, dx \sim |B|W^\qc(\Id)\, dx\sim \int_B
W^\qc(Du)\, dx\,.  
\end{equation*}
This estimate can be made precise if $Du$ is uniformly close to the identity and $W$ is 
uniformly continuous. Both properties are true only in appropriate subsets,
and hence the main part of this work is devoted to the treatment of the 
exceptional sets, under the assumption that $\|Du-\Id\|_{L^p(B)}$ is small.

In the incompressible case, since $\varphi$ is a volume-preserving bilipschitz
map a change of variables shows that the smallness of $Du-\Id$ in $L^p$ immediately translates into the 
smallness of $(Du-\Id)\circ \varphi$ in $L^p$ and hence of the contribution 
of the set where $Dz-D\varphi$ is large. This renders the volume-preserving construction in Lemma \ref{lemmaconstr1} simpler than the corresponding one in the orientation-preserving case.

In the orientation-preserving case the situation is more complex, since a
factor $\det D\varphi$ arises from the change of variables formula. One of the problematic terms is
\begin{equation*}
   \int_{\{|Du-\Id|\circ \varphi >\eps\}} |Du-\Id|^p (\varphi(x))  dx
  =  \int_{\{|Du-\Id| >\eps\}}\frac{ |Du-\Id|^p(y)}{\det D\varphi(y)} dy \,.
\end{equation*}
The fact that  $\varphi$ is Lipschitz provides a bound on $\det D\varphi$ but
not on its inverse. However, by the very same change of variables formula we
know that $1/\det D\varphi\in L^1$, therefore the above expression has the
form of the integral of the product of two $L^1$ functions. This is 
in general not  defined, but can be controlled if one of the two functions is
first shifted by an ``appropriate amount'', as done for example when  taking the
convolution of two $L^1$ functions. Precisely, this means that for every choice
of $f,g\in L^1(B_1)$ there are many $\xone\in B_{1/2}$ such that
\begin{equation*}
  \int_{B_{1/2}} f(x) g(x-\xone) dx \le  \|f\|_{L^1(B_1)} \|g\|_{L^1(B_1)}\,,
\end{equation*}
see Lemma \ref{lemmachoicex1} below for details.

In both cases the local construction is then extended to a global one by a covering
argument, see Lemma \ref{lemmarecoverydp} below.

\section{Construction of orientation-preserving maps}
\label{secorientationpres}
\subsection{Local construction}
Before presenting the construction we show how the translation is exploited. 
We focus here on the derivation of the estimates, and address measurability and weak differentiability issues in Appendix \ref{secappchainrule}.
\begin{lemma}\label{lemmachoicex1}
  Let $\psi\in W^{1,\infty}(B_r; \overline{B_r})$,   $g\in L^1(B_{r})$, $f\in L^1(B(x_0, 2r))$ for some $x_0\in \R^n$, $r>0$. Then there 
exists a measurable set $E\subset B(x_0,r)$ of positive $\calL^n$  measure 
with the following property. 
For $\xone\in E$ 
  the function 
  \begin{equation*}
    \widetilde f(x)=f(\psi(x-\xone) + \xone)g(x-\xone)
  \end{equation*}
belongs to $L^1(B(\xone,r))$ with
  \begin{equation*}
    \|\widetilde f\|_{L^1(B(\xone,r))} \le  \frac{1}{|B_r|} \|f\|_{L^1(B(x_0,2r))} \|g\|_{L^1(B_{r})} \,.
  \end{equation*}
\end{lemma}
\begin{proof}
We can assume without loss of generality that $f,g\ge0$. 
The function $(x,\xone)\mapsto \widetilde f(x)$ is $\calL^{2n}$-measurable by Lemma \ref{lemmacont}. 
 We define
$h:B(x_0,r)\to\R\cup\{\infty\}$ by 
  \begin{equation*}
    h(\xone)=\int_{B(\xone,r)} \widetilde f(x)\, dx=
\int_{B(\xone,r)} f(\psi(x-\xone)+\xone)g(x-\xone)\,dx
  \end{equation*}
  and change variables to
  \begin{alignat*}1
    h(\xone)&=\int_{B_r} f(\psi(x')+\xone) g(x')dx'\,.
  \end{alignat*}
  We integrate over all $\xone\in B(x_0,r)$ and interchange the order of integration to obtain
\begin{alignat*}1
  \int_{B(x_0,r)} h(\xone) d\xone& =\int_{B_r}\left(\int_{B(x_0,r)} 
f(\psi(x')+\xone) g(x') d\xone \right)dx'\\
&\le \|f\|_{L^1(B(x_0,2r))} \|g\|_{L^1(B_r)}\,.
\end{alignat*}
To conclude we observe that $h$ cannot be almost everywhere larger than its average.
\end{proof}
 
\begin{lemma}\label{lemmaconstr2}
Assume that the function   $W\in
C^0(\R^{n\times n}_+,[0,\infty))$  satisfies
the growth condition (\ref{eqgrowthwpd}) with $p\geq 1$ and the structure condition
(\ref{eqasstheta}) and fix 
   $F\in \R^{n\times n}_+$ and $\eta>0$. Then there is $\delta>0$ such that for any 
   $B=B(x_0,r)$ and $u\in W^{1,p}(B,\R^n)$ with 
  \begin{equation}\label{eqlemmacostr2deltadef}
    \strokedint_B \left(|Du-F|^p+|\theta(\det Du)-\theta(\det F)|\right)dx
\le \delta \text{ and } \det Du>0 \text{ a.e.}
  \end{equation}
  there are $\xone\in B(x_0,r/2)$ 
  and $z\in W^{1,p}(B,\R^n)$ with
$\det Dz>0$ a.e.,
 $z=u$ on $B(x_0,r)\setminus B(\xone,r/2)$ and
  \begin{equation}\label{eqlemmacostrrisb1}
    \int_{B(\xone,r/2)} W(Dz)\, dx\le     \int_{B(\xone,r/2)}
    (W^\qc(Du)+ \eta) \, dx  \,.
  \end{equation}
  Additionally,
  \begin{equation*}
    \int_B |u-z|^pdx \le 
    c r^p  \int_B (W^\qc(Du)+1)dx\,.
  \end{equation*}
  If $u$ is Lipschitz, then the same is true for $z$.
\end{lemma}
\begin{proof} 
The $L^p$ bound follows from the bound on $W(Dz)$
  using Poincar\'e and the growth   condition, hence we only need to prove (\ref{eqlemmacostrrisb1}). 
   
By the definition of $W^\qc(F)$ there is 
  $\varphi_\eta\in W^{1,\infty}(B_{r/2},\R^n)$ such that
  $\varphi_\eta(x)=Fx$ on 
  $\partial B_{r/2}$ and
  \begin{equation}\label{eqdefvarphieta}
    \strokedint_{B_{r/2}} W(D\varphi_\eta)dx \le W^\qc(F) + \eta\,.    
  \end{equation}
  By the growth condition  (\ref{eqgrowthwpd}) we have
  $\theta(\det D\varphi_\eta)\in L^1(B_{r/2})$, and with (\ref{eqasstheta})
also  $\theta(\det (F^{-1}D\varphi_\eta))\in L^1(B_{r/2})$.
Since  $\det D\varphi_\eta>0$ almost everywhere
 there is $\gamma>0$ (depending on $F$ and $\eta$) such that
  \begin{alignat}1\nonumber
&\int_{B_{r/2}\cap\{\det D\varphi_\eta<\gamma\}} (1+\theta(\det (F^{-1}D\varphi_\eta))) dx \\
\label{eqdefagmma}
&\hskip1cm \le  \frac{|B_{r/2}|  \eta}{(1+\|F^{-1}D\varphi_\eta\|_{L^\infty}^p)(1 + |F|^p + \theta(\det F))}\,.
  \end{alignat}
The choice of the constant on the right-hand side will become clear after
(\ref{eqwdzomegad}).

  The function $F^{-1}\varphi_\eta$ is Lipschitz continuous and therefore, by  
\cite[Theorem 1]{Ball1981}, $F^{-1}\varphi_\eta(B_{r/2})\subset \overline{B_{r/2}}$. 
For some
$\xone\in B(x_0,r/2)$ chosen below, we construct the function $z:B=B(x_0,r)\to\R^n$ by 
  \begin{equation*}
    z(x)=
    \begin{cases}
      u(F^{-1}\varphi_\eta(x-\xone)+\xone) &\text{ if } x\in B'=B(\xone,r/2)\,, \\
      u(x) &\text{ otherwise.}
    \end{cases}
  \end{equation*}
By Lemma \ref{lemmachainrule} 
(with  $\psi=F^{-1}\varphi_\eta$), there exists a null set $N$ such 
for all  choices of $\xone\not\in N$ the first expression belongs to $W^{1,1}$. 
Further we can compute its weak derivative by the usual chain rule, 
and  the  traces on $\partial B'$ of the two expressions coincide. 
In particular, $z\in W^{1,1}(B';\R^n)$. 
In order to obtain an estimate on the derivative we choose $\xone\in E\setminus N$ via 
Lemma \ref{lemmachoicex1}, applied to the ball
  $B$ with
  $f=|Du-F|^p+|\theta(\det Du)-\theta(\det F)|$ and $g=1+\theta(\det(F^{-1}D\varphi_\eta))$.  Then 
  \begin{equation}\label{eqlemmaghd}
    \strokedint_{B'} (1+\theta(\det Dv))\,  (|Du-F|^p+|\theta(\det Du)-\theta(\det F)|)\circ v \, dx \le c_{\eta}\delta \,,
  \end{equation}
where  $v(x)=F^{-1}\varphi_\eta(x-\xone)+\xone$ and
\begin{equation*}
c_{\eta}=2^n \strokedint_{B_{r/2}} (1+\theta(\det F^{-1}D\varphi_\eta))\, dx
\end{equation*}
(since $W$ and $F$ are fixed for the entire proof, we emphasize the dependence of the constants on $\eta$).
For the rest of the proof we only need to deal with the
fixed inner ball $B'$.

  Let $R_\eta=\|Dv\|_{L^\infty}$, $M_\eta=\|D\varphi_\eta\|_{L^\infty}$. 
Since $W$ is continuous   in  $\R^{n\times n}_+$ there
  is  $\varepsilon\in(0,1)$  such that
  \begin{align}\begin{aligned}
    |W(\xi)-W(\zeta)|\le \eta \text{ for all } &\xi,\zeta\in 
\R^{n\times n}_+ \text{ with }  |\zeta|\le M_\eta, \\
&\det\zeta\ge\gamma \text{ and }
|\xi-\zeta|\le  \varepsilon R_\eta\,,     
\end{aligned}
\label{eqcontwabr}
  \end{align}
where $\gamma$ was defined in (\ref{eqdefagmma}) and depends  only on $W$, $F$ and $\eta$.
Moreover, $W^\qc$ is continuous in  $\R^{n\times n}_+$. This is proven for example in 
\cite[Th. 2.4 and Prop. 2.3]{Fonseca1988} by showing that $W^\qc$ is 
rank-one convex, and hence separately convex, in the open set $\R^{n\times n}_+$.
Hence we may assume additionally that
  \begin{equation}\label{eqcontwAF}
    |W^\qc(\xi)-W^\qc(F)|+|\theta(\det\xi)-\theta(\det F)|\le \eta \text{ for all  $\xi$ with }  |\xi-F|\le
    \varepsilon \,.
  \end{equation}
  The parameter $\varepsilon$ depends on $\eta$,
  but not   on $u$ and  $\delta$. 
  We compute, with $\widehat\varphi_\eta(x)=\varphi_\eta(x-\xone)$,
  \begin{alignat}1\nonumber
    \int_{B'} (W(Dz)-W^\qc(Du))dx
=&    \int_{B'} (W(Dz)-W(D\widehat\varphi_\eta))dx\\
\nonumber
&+\int_{B'}(W(D\widehat\varphi_\eta)-W^\qc(F))dx \\
&+\int_{B'}(W^\qc(F)-W^\qc(Du))dx   \label{eqwdzwqcduing}
  \end{alignat}
and estimate the three terms separately.

  The second term in (\ref{eqwdzwqcduing}) is bounded by $\eta|B'|$ by the
  definition of   $\varphi_\eta$, see (\ref{eqdefvarphieta}).
  To treat the third one we use  (\ref{eqcontwAF}) to obtain 
  \begin{equation*}
    W^\qc(F)\le W^\qc(Du)+ \eta 
\hskip1cm\text{ on the set where $|Du-F|\le \eps$}  \,.
  \end{equation*}
The complement is small, indeed, from 
 (\ref{eqlemmacostr2deltadef}) we obtain
  \begin{alignat*}1
    \int_{B'} (W^\qc(F)-W^\qc(Du))dx &\le \eta |B'| + W^\qc(F)
    \calL^n(|Du-F|>\eps)\\
&    \le \eta |B'| +  W^\qc(F) \frac{1}{\eps^p}  |B| \delta  \,.
  \end{alignat*}
To estimate the first term in (\ref{eqwdzwqcduing}) we distinguish three subsets:
 $\omega=B'\cap \{|Du-F|\circ v \ge\varepsilon\}$,
$\omega_d=B'\cap\{\det D\widehat\varphi_\eta<\gamma\}\setminus\omega$ 
and the rest $B'\setminus\omega\setminus\omega_d$.

In  $B'\setminus\omega\setminus\omega_d$ we have $|Dv|\le R_\eta$, $\det D\widehat\varphi_\eta\ge\gamma$, $|Du-F|\circ 
v\le \eps$ and therefore, since
  \begin{equation*}
    Dz = Du \circ v Dv = (Du-F)\circ v Dv + D\widehat\varphi_\eta 
  \end{equation*}
we obtain
\begin{equation*}
    |Dz-D\widehat\varphi_\eta|\le |Du-F|\circ v \, |Dv| \le \eps R_\eta\,.
\end{equation*}
By the continuity estimate (\ref{eqcontwabr}) we obtain
  \begin{equation*}
    |W(Dz)-W(D\widehat\varphi_\eta)|    \le \eta
  \end{equation*}
and therefore
\begin{equation*}
  \int_{B'\setminus\omega_d\setminus\omega} (W(Dz)-W(D\widehat\varphi_\eta)) dx\le \eta|B'|\,.
\end{equation*}

In the two error sets we use the growth estimate (\ref{eqgrowthwpd}), which gives
\begin{alignat}1\label{eqwdzfehlerter1}
  W(Dz)&\le c (1+|Du|^p\circ v \, |Dv|^p + \theta(\det Du\circ v\, \det Dv))\,.
\end{alignat}
With $|Dv|\le R_\eta$  and (\ref{eqasstheta}) we obtain
\begin{alignat}1\label{eqwdzfehlerter}
  W(Dz)&\le c (1+R_\eta^p |Du|^p\circ v  + (1+\theta(\det Du)\circ v) \,(1+\theta(\det Dv))\,,
\end{alignat}
where $c$ only depends on $W$.
At this point we treat the two error sets separately. For the estimate on
$\omega$ we observe that
 $|Du-F| \ge\varepsilon$ implies
  \begin{equation*}
    |Du|+1\le |Du-F|+|F|+1 \le \left( \frac{|F|+1}{\eps}+1\right) |Du-F|
  \end{equation*}
and
\begin{equation*}
  \theta(\det Du)\le |\theta(\det Du)-\theta(\det F)|+ \frac{\theta(\det F)}{\eps^p}|Du-F|^p\,.
\end{equation*}
Therefore (\ref{eqwdzfehlerter}) gives
  \begin{alignat*}1
  \int_\omega W(Dz)&\le
  c \int_{\omega} (1+\theta(\det Dv))(1+R_\eta^p|Du|^p+\theta(\det Du))\circ v
\, dx\\ 
\le
  c_{\eta}& \int_{\omega} (1+\theta(\det Dv)) (|Du-F|^p+|\theta(\det Du)-\theta(\det F)|)\circ v\, dx\\
\le   c_{\eta}& |B'| \delta\,.
  \end{alignat*}
where in the last step we used  (\ref{eqlemmaghd}).
The constant depends on $W$, $F$ and $\eta$ (via $\eps$), but not
  on $\delta$ and $u$.

In $\omega_d$ instead we have $|Du-F|\circ v\le \eps$. 
Then, recalling the continuity estimate (\ref{eqcontwAF}), we have  $|Du|\circ v\le |F|+1$ and
$\theta(\det Du\circ v)\le \theta(\det F)+1$ and therefore
(\ref{eqwdzfehlerter}) reduces to 
\begin{alignat}1\nonumber
  W(Dz)&\le c (1+ R_\eta^p (1+|F|^p)  + (1+\theta(\det F)) (1+\theta(\det Dv))\\
&\le c_*(1+\|F^{-1}D\varphi_\eta\|_\infty^p) 
(1+|F|^p+\theta(\det F))(1+\theta(\det Dv))
\,,
\label{eqwdzomegad}
\end{alignat}
with a constant $c_*>0$ which depends only on $W$.
With  (\ref{eqdefagmma}) we conclude
\begin{equation*}
  \int_{\omega_d} W(Dz)dx\le c_* |B'| \eta\,.
\end{equation*}
Adding all terms we obtain
\begin{equation*}
  \int_{B'} \left(W(Dz)-W^\qc(Du)\right)dx \le  (
3\eta +  2^n\frac{W^{\qc}(F)}{\eps^p} \delta +
 c_{\eta} \delta + c_*\eta) |B'|\,.
\end{equation*}
Since $c_*$ depends only on $W$, choosing $\delta$ sufficiently small the proof is concluded.
  \end{proof}

\subsection{Upper bound}
  \begin{lemma}[Recovery sequence]\label{lemmarecoverydp}    Let $\Omega\subset\R^n$ open, Lipschitz, bounded, 
let $W\in C^0(\R^{n\times n}_+,[0,\infty))$  obey
(\ref{eqgrowthwpd}--\ref{eqasstheta}). Then for any
$u\in
    W^{1,p}(\Omega;\R^n)$ with $\det Du>0$ almost everywhere there is 
    a sequence $u_j\rightharpoonup u$ in $W^{1,p}$ such that $\det Du_j>0$
    almost everywhere, $u_j=u$ on $\partial\Omega$ and 
    \begin{equation*} 
      \limsup_{j\to\infty}\int_\Omega W(Du_j)dx\le \int_\Omega W^\qc(Du)dx\,.      
    \end{equation*}
  If additionally $u\in W^{1,\infty}$ then also $u_j\in
    W^{1,\infty}$. 
  \end{lemma}
  \begin{proof}
    Fix $\eta>0$. It suffices to construct $w$ with $\|u-w\|_p\le \eta$,
$w=u$ on $\partial\Omega$,
    and $\int_\Omega W(Dw)dx\le \int_\Omega W^\qc(Du)dx+\eta$.

    If $\int_\Omega W^\qc(Du)dx=\infty$ the constant sequence will do, hence we
    can assume that $W^\qc\circ Du\in L^1$.
    By convexity of $\theta$ and of the $p$-norm the definition of $W^\qc(F)$ gives 
    \begin{equation*}
      \frac1c |F|^p +\frac1c\theta(\det F)-c\le W^\qc(F)\text{ for all $F$},
    \end{equation*}
    therefore $|Du|^p$ and
    $\theta(\det Du)$ are also integrable. We denote by  $E$ the set of
    Lebesgue points of $Du$ and $\theta(\det Du)$.
    For every $x\in E$  we set $F(x)=Du(x)$  and choose $\delta(x)$ as in
    Lemma \ref{lemmaconstr2} 
    for this $F$ and $\eta$ as above.

    The construction is done by successive application of 
    Lemma \ref{lemmaconstr2}. We set $w_0=u$, $\Omega_0=\Omega$ and describe
    how to pass from   $(w_j,\Omega_j)$ to $(w_{j+1},\Omega_{j+1})$.
    For all $x\in E\cap \Omega_j$ we choose
     $r_j(x)\in(0,\eta)$ such that $B(x,r_j(x))\subset \Omega_j$ and
    \begin{equation*}
      \strokedint_{B(x,r)} \left(|Dw_j-F(x)|^p+|\theta(\det Dw_j)-\theta(\det F(x))|\right) 
dx' \le \delta(x) 
    \end{equation*}
for all $r<r_j(x)$. 
    This gives a fine cover of $E\cap\Omega_j$. We extract a disjoint subcover
    $B(x_k,r_k)_{k\in\N}$ and from this subcover finitely many balls
    $B(x_k,r_k)_{k=0,\dots, M}$ which cover at least half the volume of
    $\Omega_j$. 

    We set $w_{j+1}=w_j$ on $\Omega\setminus \cup_{k=0}^M B(x_k,r_k)$
    and define $w_{j+1}$ as the result of  Lemma \ref{lemmaconstr2} in each of
    the 
    balls. Then $w_{j+1}\in W^{1,p}(\Omega;\R^n)$ and $w_{j+1}=w_j=u$
    on $\partial\Omega$. Further, 
    the smaller  balls $B(x_k',r_k/2)\subset B(x_k,r_k)$ obey
    \begin{equation}\label{eqwdwj1}
      \int_{B(x_k',r_k/2)} W(Dw_{j+1})dx\le
      \int_{B(x_k',r_k/2)} (W^\qc(Du)+\eta) dx
    \end{equation}
and
    \begin{equation}\label{eqconvlp}
      \int_{B(x_k',r_k/2)} |w_{j+1}-u|^pdx\le c \eta^p 
      \int_{B(x_k',r_k/2)} (1+W^\qc(Du))dx\,,
    \end{equation}
    with $w_{j+1}=w_j$ outside these balls.
    Finally we set $\Omega_{j+1}=\Omega_j\setminus  \cup_{k=0}^M
    \overline B(x_k',r_k/2)$, so that $|\Omega_{j+1}|\le (1-2^{-n-1}) |\Omega_j|$,
    and  iterate. We remark that $w_{j+1}=u$ on the open set $\Omega_{j+1}$,
    hence there is no need to redefine $E$, $F$ and $\delta$ at each step.
    This concludes the construction of the sequence $w_j$.

It remains to show that $w_j$, for $j$ sufficiently large, has the desired
properties. Each of these functions coincides with $u$ outside a finite number
of disjoint balls, and has been modified exactly once in each of those balls.
By (\ref{eqconvlp}) we have
\begin{equation*}
  \int_{\Omega} |w_j-u|^pdx \le c \eta^p \int_\Omega (1+W^\qc(Du))dx
\end{equation*}
hence $w_j$ is close to $u$ in $L^p$, independently of $j$.

Analogously from (\ref{eqwdwj1}) we deduce, for the union of the balls
$\Omega\setminus \Omega_j$,
    \begin{equation*}
      \int_{\Omega\setminus\Omega_{j}} W(Dw_{j})dx\le
      \int_{\Omega\setminus\Omega_{j}} (W^\qc(Du)+\eta)dx
    \end{equation*}
which implies
\begin{equation*}
  \int_\Omega W(Dw_j)dx\le \int_{\Omega\setminus\Omega_j} (W^\qc(Du)+\eta)dx
  +\int_{\Omega_j} W(Du)dx\,.
\end{equation*}
Since  $|\Omega_j|\le (1-2^{-n-1})^j|\Omega|\to0$ and by the growth condition
$W(Du)\in L^1(\Omega)$, for
sufficiently large $j$ we have
\begin{equation*}
  \int_\Omega W(Dw_j)dx\le \int_{\Omega} (W^\qc(Du)+2\eta)dx\,,
\end{equation*}
as required.
  \end{proof}

  \begin{lemma}[Quasiconvexity]\label{lemmaqcorientpres}
    The function $W^\qc$ is quasiconvex. 
  \end{lemma}
  \begin{remark}\label{remarkwqc}
  $W^\qc$ is the largest (extended-valued) quasiconvex function
below $W$, hence in this sense its quasiconvex envelope. This function does not necessarily coincide with the supremum of all finite-valued quasiconvex functions below $W$; in particular, this is not true for the function discussed in \cite[Example 3.5]{BallMurat1984}.
      \end{remark}
  \begin{proof}
    Fix $F$ with $\det F>0$, $\Omega=B_1$, $\psi\in W^{1,\infty}(B_1,\R^n)$
    with $\psi(x)=Fx$ on $\partial B_1$. We need to show that 
    \begin{equation*}
      W^\qc(F)\le \strokedint_{B_1} W^\qc(D\psi)dx\,.
    \end{equation*}
By Lemma \ref{lemmarecoverydp}
    there is a sequence $\varphi_j\in W^{1,\infty}$ 
    with $\varphi_j(x)=\psi(x)=Fx$ on $\partial B_1$ and such that
    \begin{equation*}
      \limsup_{j\to\infty} \int_{B_1} W(D\varphi_j)dx\le \int_{B_1}
      W^\qc(D\psi)dx\,. 
    \end{equation*}
    Since every $\varphi_j$ is admissible in the definition of $W^\qc(F)$ we
    obtain
    \begin{equation*}
      W^\qc(F)\le \strokedint_{B_1} W(D\varphi_j) dx
    \end{equation*}
    for all $j$, and in particular
    \begin{equation*}
      W^\qc(F)\le\strokedint_{B_1}
      W^\qc(D\psi)dx
    \end{equation*}
    as desired.
  \end{proof}

\section{Construction of volume-preserving maps}
 \label{secvolumepres}
In this case the translation is not needed, and correspondingly the proof of Lemma \ref{lemmaconstr1} is simpler than the one of Lemma \ref{lemmaconstr2}; we give it in detail since it illustates in a compact way the key ideas of our constrruction. At the same time the continuity of $W^\qc$ is less clear than for orientation-preserving maps. 
It essentially follows from the results of \cite{MuellerSverak,Conti2008}. Since it was not stated there we briefly show how it can be derived from the construction in \cite{Conti2008}.

\begin{lemma}
Given $W:\Sigma\to[0,\infty)$, extended by $\infty$ elsewhere, we define
$W^\qc$  by (\ref{eqdefwqcdetp}).
  The function $W^\qc$ is rank-one convex and hence continuous on $\Sigma$.
\end{lemma}
\begin{proof}
We first observe that, by general scaling and covering arguments, the definition of $W^\qc$ does not depend on the domain, and in particular
\begin{equation}\label{eqwqcpolyhedron}
  W^\qc(F)=\inf\{\strokedint_{\omega} W(D\varphi)dx: \varphi\in
  W^{1,\infty}(\omega;\R^n), \varphi(x)=Fx \text{ for } x\in \partial \omega \}
\end{equation}
for any bounded open nonempty polyhedron $\omega\subset\R^n$.

To prove  rank-one convexity we fix $A,B\in \Sigma$ with
  $\rank(A-B)=1$ and $\lambda\in(0,1)$. We define $F=\lambda A+  (1-\lambda) B$.
By the construction in \cite[Th. 2.1]{Conti2008} 
(with $n=m=r$, $P=Q=\Id$, $t=1$, $\eps=1$) there is a finite
set $K\subset\Sigma$ such that for any $\delta>0$ one can find 
a polyhedron $\Omega$ and a piecewise affine function $u\in
  W^{1,\infty}(\Omega;\R^n)$ such that $u(x)=Fx$ on $\partial \Omega$,
  $Du \in K\subset\Sigma$ almost everywhere, 
 $|\{Du\not\in\{A,B\}\}|\le \delta|\Omega|$.
 The latter, together with the boundary data, implies
 \begin{equation*}
|\{Du=A\}|\le (\lambda+c\delta)|\Omega| \text{ and }
|\{Du=B\}|\le (1-\lambda+c\delta)|\Omega|\,,
 \end{equation*}
with $c$ depending on $A$ and  $B$.
Further, the set $\{Du=A\}$ is a finite union of simplexes $\omega^A_j$. For
each of 
 them there is, by (\ref{eqwqcpolyhedron}) with $F=A$,  a Lipschitz function $v^A_j$ 
 with $v^A_j=u$ on $\partial \omega^A_j$ and
 \begin{equation*}
   \int_{ \omega^A_j} W(Dv^A_j)dx \le  |\omega^A_j|(  W^\qc(A) + \delta)\,.
 \end{equation*}
The same holds for the set $\{Du=B\}$. We set $w=v^A_j$ on each
$\omega^A_j$,  $w=v^B_j$ on each
$\omega^B_j$,  $w=u$ on the rest. 
Since $w(x)=Fx$ on $\partial\Omega$  we have
\begin{alignat*}1
 W^\qc(F)\le & \strokedint_\Omega W(Dw)dx \le \frac{|\{Du=A\}|}{|\Omega|}(  W^\qc(A) + \delta)\\
& +\frac{|\{Du=B\}|}{|\Omega|}(  W^\qc(B) + \delta)
+\frac{|\{Du\not\in \{A,B\}\}|}{|\Omega|} \max W(K)\\
&\le \lambda W^\qc(A)+(1-\lambda)W^\qc(B) + c\delta \max W(K)\,,
\end{alignat*}
with $c$ depending on $A$ and $B$. 
Taking  $\delta\to0$ (with fixed $K$) this implies  the desired inequality
$W^\qc(F)\le \lambda W^\qc(A)+(1-\lambda) W^\qc(B)$.
Since $W^\qc$ is rank-one convex, it is separately
convex in suitable variables and hence continuous (for details see, e.g., 
  \cite[Step 2 in the proof of Th. 3.1]{Conti2008}).
\end{proof}

\begin{lemma}\label{lemmaconstr1}
  Let 
$W\in C^0(\Sigma;[0,\infty))$ obey
(\ref{eqgrowthwincompr}) for some $p\ge 1$.
  Then for any $F\in \R^{n\times n}$ and $\eta>0$ 
there is $\delta>0$ such that the following holds:
For any ball $B=B(x_0,r)$ and any function  $u\in W^{1,p}(B,\R^n)$
with 
  \begin{equation*}
    \strokedint_B |Du-F|^pdx\le \delta \text{ and } Du\in \Sigma \text{ a.e.}
  \end{equation*}
  one can find $z\in W^{1,p}(B,\R^n)$ with $u=z$ on $\partial B$,
  \begin{equation*}
\strokedint_B W(Dz)dx\le     \strokedint_B (W^\qc(Du)+ \eta)dx  \text{ and } Dz
\in \Sigma \text{ a.e.}
  \end{equation*}
  Additionally,
  \begin{equation*}
    \strokedint_B |u-z|^p dx\le 
    c r^p  \strokedint_B (1+W^\qc(Du))dx\,.
  \end{equation*}
\end{lemma}
\begin{proof}
  Let $\varphi_\eta\in W^{1,\infty}(B,\R^n)$ be such that $\varphi_\eta(x)=Fx$ on
  $\partial B$ and
  \begin{equation*}
    \strokedint_B W(D\varphi_\eta)dx \le  W^\qc(F) + \eta\,.    
  \end{equation*}
  We define
  \begin{equation*}
    v=F^{-1}\varphi_\eta
  \end{equation*}
  and observe that, by \cite[Theorem 2]{Ball1981},
$v$ is a bilipschitz
  map from $B$ onto itself. Therefore we can define
  \begin{equation*}
    z=u\circ v \in W^{1,p}(B,\R^n)
  \end{equation*}
  and compute its gradient
  \begin{equation*}
    Dz = Du \circ v Dv = (Du-F)\circ v Dv + D\varphi_\eta\,.
  \end{equation*}
  We set $R_\eta=\|Dv\|_\infty$, $M_\eta=\|D\varphi_\eta\|_\infty$ and choose $\varepsilon\in(0,1)$ such that
  \begin{equation}\label{eqcontinc1}
    |W(\xi)-W(\zeta)|\le \eta \text{ whenever }  |\zeta|\le M_\eta\text{ and } |\xi-\zeta|\le
    \varepsilon R_\eta
  \end{equation}
  and
  \begin{equation}\label{eqcontinc2}
    |W^\qc(\xi)-W^\qc(F)|\le \eta \text{ whenever }  |\xi-F|\le   \varepsilon \,. 
  \end{equation}

In order to estimate the integral
  \begin{alignat}1\nonumber
    \int_B (W(Dz)-W^\qc(Du))dx
=&    \int_B (W(Dz)-W(D\varphi_\eta))dx\\ \nonumber
&+\int_B(W(D\varphi_\eta)-W^\qc(F))dx \\
&+\int_B(W^\qc(F)-W^\qc(Du))dx
\label{eqedzdfincomprs}  
  \end{alignat}
we consider the three terms separately.
  The second integral in (\ref{eqedzdfincomprs}) is bounded by $\eta |B|$ by the definition of
  $\varphi_\eta$.
In order to estimate the last integral in (\ref{eqedzdfincomprs}) we use
(\ref{eqcontinc2}) to obtain
  \begin{equation*}
    W^\qc(F)\le W^\qc(Du)+ \eta 
    \hskip1cm\text{ on the set where $|Du-F|\le \eps$}  \,.
  \end{equation*}
  The complement is small and gives a small contribution. Precisely,
  \begin{alignat*}1
    \int_{B} (W^\qc(F)-W^\qc(Du))dx &\le \eta |B| + W^\qc(F)
    \calL^n(|Du-F|>\eps)\\
&    \le \eta |B| +  W^\qc(F) \frac{1}{\eps^p}  |B| \delta  \,.
  \end{alignat*}
In order to estimate the  first integral in (\ref{eqedzdfincomprs}) we
distinguish the set
 $\omega=\{|Du-F|\circ v >\varepsilon\}$ and the rest. On $B\setminus\omega$, 
from the explicit expression for $Dz$ we obtain
\begin{equation*}
  |Dz-D\varphi_\eta| \le \|Dv\|_\infty |Du-F|\circ v \le \eps R_\eta
\end{equation*}
and recalling (\ref{eqcontinc1}) we can estimate 
  \begin{equation*}
    |W(Dz)-W(D\varphi_\eta)| \le \eta \text{ on } B\setminus \omega\,.
  \end{equation*}
  Since $|Du-F|\ge \eps$ implies
  \begin{equation*}
    |Du|+1\le |Du-F|+|F|+1 \le \left( \frac{|F|+1}{\eps}+1\right) |Du-F|\,,
  \end{equation*}
the contribution of $\omega$ can be estimated by
  \begin{alignat*}1
    \int_\omega W(Dz)dx \le
    c\int_{\omega} (R_\eta^p |Du|^p\circ v +1) dx
\le        &  c_{F,\eps} \int_{\omega} |Du-F|^p\circ v \, dx
  \end{alignat*}
  where the constant depends on $F$, $\eta$ and $\eps$. Finally, 
$v$ is a bilipschitz
  map from $B$ onto itself with $\det Dv=1$ almost everywhere
(see \cite[Theorem 2]{Ball1981}) and therefore
  \begin{alignat*}1
    \int_\omega W(Dz)dx\le
        &  c_{F,\eps} \int_{B} |Du-F|^p\circ v  \, dx
=          c_{F,\eps} \int_{B} |Du-F|^p dx\le c_{F,\eps} \delta |B|\,.
  \end{alignat*}
Collecting terms we conclude
\begin{equation*}
  \strokedint_B (W(Dz)-W^\qc(Du))dx\le 3\eta +
 \frac{W^\qc(F)}{\eps^p}\delta
+      c_{F,\eps}\delta\,.
\end{equation*}
Since $\eps$ depends on $\eta$ and $F$ but not on $\delta$ and $u$, choosing $\delta$ sufficiently small the
proof is concluded. The $L^p$ estimate follows from the growth estimate and Poincar\'e's inequality.
  \end{proof}

  \begin{lemma}[Recovery sequence]\label{lemmarecoveryincompr}
    Let $\Omega\subset\R^n$ open, Lipschitz, bounded, $u\in
    W^{1,p}(\Omega;\R^n)$ with $Du\in \Sigma$ almost everywhere. Then there is
    a sequence $u_j\rightharpoonup u$ in $W^{1,p}$ such that $\det Du_j\in \Sigma$
    almost everywhere and
    \begin{equation*} 
      \limsup_{j\to\infty}\int W(Du_j)dx\le \int W^\qc(Du)dx\,.      
    \end{equation*}
  If additionally $u\in W^{1,\infty}$ then also $u_j\in
    W^{1,\infty}$. 
  \end{lemma}
  \begin{proof}
The proof is just like the one in Lemma \ref{lemmarecoverydp}, for brevity we do not repeat it.
  \end{proof}

  \begin{lemma}[Quasiconvexity]
    The function $W^\qc$ is quasiconvex.    
  \end{lemma}
  \begin{proof}
The proof is just like the one of Lemma \ref{lemmaqcorientpres}, for brevity we do not repeat it.
  \end{proof}

\appendix
\section{Composition of Sobolev functions  with Lipschitz functions}
\label{secappchainrule}
The composition of a Lipschitz with a Sobolev function and the composition of a Sobolev with a bilipschitz function are standard.
Although there is a substantial literature on the subject, see for example \cite{GoldsteinReshetnyak1990,LeoniMorini2007} and references therein, we have been unable to find the statement needed here on the composition of a Sobolev with a Lipschitz function, hence we give a short self-contained proof.
To see the difficulty with measurability one can consider the example $\psi(x_1,x_2)=(x_1,0)$, 
$f(x_1,x_2)=h(x_1)\chi_{\{0\}}(x_2)$, with $h:\R\to\R$ not measurable. Then $f=0$ $\calL^2$-almost everywhere but $f\circ\psi$ is not measurable.  To see the difficulty with integrability one can consider $f(x)=|x|^{-1/2}$ in the unit ball of $\R^2$, with $\psi(x)=|x|x$ around the origin. Then $f$ is in $W^{1,1}$ but $(f\circ \psi)(x)=|x|^{-1}$ is not.
\def\gpsi{g}
\begin{lemma}\label{lemmacont}
  Let $\psi\in W^{1,\infty}(B_1;\overline B_1)$, 
$f_k\in L^1(B_2;\R^m)$ with   $\sum_k\|f_k\|_{L^1(B_2)}<\infty$.
Then the maps
$(x,\xone)\mapsto f_{k}(\xone+\psi(x-\xone))$ are $\calL^{2n}$ measurable and 
for almost all  $\xone\in B_1$ 
 the functions
\begin{equation*}
z_{k}(x)=f_{k}(\xone+\psi(x-\xone))
\end{equation*}
are in $L^1(B(\xone,1))$ with $\sum_k \|z_{k}\|_{L^1(B(\xone,1))}<\infty$.
\end{lemma}
\begin{proof}
We define the continuous function $\gpsi:B_1\times B_1\to B_2$  by $\gpsi(x,y)=x+\psi(y)$ and show that for any $k$ the function $f_k\circ \gpsi$ is $\calL^{2n}$-measurable. 

Let $A\subset\R^m$ be open. Then $f_k^{-1}(A)\subset B_2$ is $\calL^n$-measurable, therefore $f_k^{-1}(A)=E\setminus N$, with $E$ Borel and $N$ a null set. 
Since $\gpsi$ is continuous, $\gpsi^{-1}(E)$ is Borel. It remains to show that $|N|=0$ implies $\gpsi^{-1}(N)=0$.  
Let $F\subset B_2$ be Borel with $N\subset F$ and $|F|=0$. 
Then $\gpsi^{-1}(F)$ is Borel and $\calL^{2n}$-measurable. For any $y\in\R^n$, the set $T_y=\{x\in \R^n: \gpsi(x,y)\in F\}=F-\psi(y)$ is a $\calL^n$-null set. By Fubini's theorem
\begin{equation*}
  \calL^{2n}(\gpsi^{-1}(F))=\int_{\R^n} \calL^n(T_y) dy = 0\,.
\end{equation*}
Therefore $f_k\circ \gpsi$ is measurable.
A second application of  Fubini's theorem shows that for almost all $\xone\in B_1$ each function $x\mapsto f_k(\gpsi(\xone,x))$ is measurable; clearly 
 the same holds for the translations 
$z_k(x)=f_k(\gpsi(\xone,x-\xone))$. We conclude that for almost all $\xone\in B_1$ all the functions $z_k$ are $\calL^n$-measurable.

We define $A:B_1\to[0,\infty]$ by
\begin{alignat*}1
  A(\xone)&=\sum_{k\in\N}   \|z_k\|_{L^1(B(\xone,1))}
=\sum_{k\in\N} \int_{B(\xone,1)}  |f_k|(\xone+\psi(x-\xone)) dx \\
&=\sum_{k\in\N} \int_{B_1}  |f_k|(\xone+\psi(x')) dx'\,.
\end{alignat*}
The integrand is nonnegative and measurable, hence we can interchange 
 the order of summation and integration. Since the integrand is measurable as a function on $\R^{2n}$, the function $A$ is measurable.
Integrating and changing variables as usual,
\begin{equation*}
  \int_{B_1} A(\xone) d\xone \le \sum_{k\in\N} \int_{B_1}  \|f_k\|_{L^1(B_2)} dx' 
  \le  |B_1| \sum_{k\in \N}\|f_k\|_{L^1(B_2)}<\infty \,.
\end{equation*}
Therefore  $A(\xone)<\infty$ almost everywhere, which concludes the proof.
\end{proof}

\begin{lemma}[Chain rule]\label{lemmachainrule}
  Let $\psi\in W^{1,\infty}(B_1;\overline B_1)$, $u\in W^{1,1}(B_2;\R^m)$.
Then for almost all  $\xone\in B_1$ 
 the function $w(x)=u(\xone+\psi(x-\xone))$ belongs to $W^{1,1}(B(\xone,1);\R^m)$ with
\begin{equation*}
  Dw(x)=Du(\xone+\psi(x-\xone))D\psi(x-\xone)\,.
\end{equation*}
If $\psi(x)=x$ on $\partial B_1$ then $w=u$ (as traces) on $\partial B(\xone,1)$.
\end{lemma}
\begin{proof}
We choose a sequence  $u_k\in C^\infty(\overline{B_2};\R^m)$ such that $\|u_k-u\|_{W^{1,1}(B_2)}\le 2^{-k}$ and
 apply  Lemma \ref{lemmacont} to the sequence 
 $f_k=(u_k-u,Du_k-Du)\in L^1(B_2;\R^{m}\times \R^{m\times n})$,
 which obeys $\sum \|f_k\|_{L^1}\le 2$. 
For any fixed $\xone$ not in the null set given by the lemma,
we obtain the corresponding sequence  $z_k$ with the properties asserted in
Lemma \ref{lemmacont}. Additionally
we define $w_k$ by $w_k(x)=u_k(\xone+\psi(x-\xone))$ and $w$ as in the statement.
Each of the functions $z_k$ with values in $\R^m\times \R^{m\times n}$ is measurable,
therefore the  first $m$ components which are given by $w_k-w$ are measurable.
The continuity of 
 $w_k$ implies the measurability of $w$. Furthermore, 
\begin{alignat*}1
\|w_k-w\|_{L^1(B(\xone,1))} &= \int_{B(\xone,1)} |u_k-u|(\xone+\psi(x-\xone)) dx\\
&  \le \int_{B(\xone,1)} |f_k|(\xone+\psi(x-\xone)) dx =\|z_k\|_{L^1(B(\xone,1))}\to0\,. 
\end{alignat*}
We conclude $w\in L^1$ and $w_k\to w$ in $L^1$.  

We now repeat the procedure for the gradient.
We denote by $F(x)=Du(\xone+\psi(x-\xone))D\psi(x-\xone)$  the expression given in the statement. 
 Since every $u_k$ is smooth  by the usual chain rule we obtain
 \begin{equation*}
     Dw_k(x) = Du_k(\xone+\psi(x-\xone))D\psi(x-\xone)\,,
 \end{equation*}
which is the product of a continuous and an $L^\infty$ function and therefore measurable. Further, 
\begin{equation*}
  (Dw_k-F)(x) = (Du_k-Du)(\xone+\psi(x-\xone))D\psi(x-\xone)\,.
\end{equation*}
The first factor is the second component of $z_k$ hence measurable by Lemma \ref{lemmacont}, the second belongs to $L^\infty$. Continuity of $Dw_k$  gives measurability of  $F$. Further,
\begin{alignat*}1
  \|Dw_k-F\|_{L^1(B(\xone,1))} &\le \|D\psi\|_\infty \int_{B(\xone,1)} |Du_k-Du|(\xone+\psi(x-\xone)) dx\\
&  \le \|D\psi\|_\infty\int_{B(\xone,1)} |f_k|(\xone+\psi(x-\xone)) dx\to0 \,.
\end{alignat*}
Therefore $F\in L^1$ and $Dw_k\to F$ in $L^1$. Continuity of the distributional derivative implies  $F=Dw$ distributionally and $w\in W^{1,1}$.  

To obtain the condition on the trace it suffices to extend $\psi$ to be the identity outside $B_1$, $u$ to a function in $W^{1,1}(\R^n;\R^m)$ and work on a larger ball.
\end{proof}

\section{Construction for submultiplicative integrands}
\label{sectmultipl}
We show here how our construction of the recovery sequence can be extended to the more general situation discussed in Remark \ref{remarkgener}. We focus on the orientation-preserving case, the other one is simpler. For brevity we only show how the basic construction step is modified, the covering of Lemma \ref{lemmarecoverydp} is not significantly changed. Indeed, it suffices to use $p=1$ and takes Lebesgue points of $Du$ and $W(Du)$ instead of Lebesgue points of $Du$ and  $\theta(\det Du)$; $W(Du)\in L^1$ by the growth condition (\ref{eqwprod2}).  
\begin{lemma}\label{lemmaconstr2prod}
Assume that  $W\in
C^0(\R^{n\times n}_+,[0,\infty))$ satisfies 
\begin{equation}\label{eqwgrowthproduct1}
\frac1c|F|-c\le  W(G)
\end{equation}
and
\begin{equation}\label{eqwgrowthproduct}
  W(FG)\le c_W (1+W(F))(1+W(G)) 
\end{equation}
for all $F,G\in \R^{n\times n}_+$, with a fixed $c_W>0$.
 Then for any
   $F\in \R^{n\times n}_+$ and $\eta>0$ there is $\delta>0$ such that for any
   $B=B(x_0,r)$ and $u\in W^{1,1}(B,\R^n)$ with 
  \begin{equation*}
    \strokedint_B ( |Du-F|+|W(Du)-W(F)|)\, dx\le \delta \text{ and } \det Du>0 \text{ a.e.}
  \end{equation*}
  there are $\xone\in B(x_0,r/2)$ 
  and $z\in W^{1,1}(B,\R^n)$ with
$\det Dz>0$ a.e.,
 $z=u$ on $B(x_0,r)\setminus B(\xone,r/2)$ and
  \begin{equation*}
    \int_{B(\xone,r/2)} W(Dz)dx\le     \int_{B(\xone,r/2)}
    (W^\qc(Du)+ \eta) dx \,.
  \end{equation*}
  Additionally,
  \begin{equation*}
    \int_B |u-z|dx \le 
    c r  \int_B (W^\qc(Du)+1)dx\,.
  \end{equation*}
  If $u$ is Lipschitz, then so is $z$.
\end{lemma}
\begin{proof}
  This is very similar to the proof of Lemma \ref{lemmaconstr2}, we only discuss the differences. 
After (\ref{eqdefvarphieta}), (\ref{eqwgrowthproduct}) implies $W(F^{-1} D\varphi_\eta)\in L^1$ and 
equation (\ref{eqdefagmma}) is replaced by
    \begin{equation}\label{eqdefagmmaprod}
\int_{B_{r/2}\cap\{\det D\varphi_\eta<\gamma\}} (1+W(F^{-1}D\varphi_\eta))\, dx \le \frac{|B_{r/2}|}{c_W(2+W(F))} \eta\,.
  \end{equation}
  In  Lemma \ref{lemmachoicex1} we use
  $f=|Du-F|+|W(Du)-W(F)|$ and $g=1+W(F^{-1}D\varphi_\eta)$,
(\ref{eqlemmaghd}) is replaced by
  \begin{equation}\label{eqlemmaghdprod}
    \strokedint_{B'} (1+W(F^{-1}D\widehat\varphi_\eta))\,  (|Du-F|+|W(Du)-W(F)|)\circ v \, dx
\le c_{\eta}\delta \,,
  \end{equation}
where $c_{\eta}= 2^n \strokedint_{B_{r/2}} (1+W(F^{-1}D\varphi_\eta))\,dx<\infty$.
In (\ref{eqcontwAF}) we use continuity of $W$ instead of $\theta$.
The remaining differences are in the treatment of the two error sets.
We replace (\ref{eqwdzfehlerter1}) by
\begin{equation}\label{eqwdzfehlerterprod}
  W(Dz) \le c_W (1+W(Du)\circ v) (1+W(Dv))\,.
\end{equation}
We start from $\omega$. 
From $|Du-F|\circ v\ge\eps$ we deduce
\begin{alignat*}1
  1+W(Du)\circ v \le& 1+W(F)+|W(Du)-W(F)|\circ v\\
 \le &c_{F,\eps} (|Du-F|+|W(Du)-W(F)|)\circ v
\end{alignat*}
where $c_{F,\eps}=1+(1+W(F))/\eps$. 
Therefore the estimate (\ref{eqwdzfehlerterprod}) gives
\begin{alignat*}1
  \int_\omega W(Dz)dx \le & c_Wc_{F,\eps}\int_{B'}  (1+W(Dv)) \, (|Du-F|+|W(Du)-W(F)|)\circ v dx
\end{alignat*}
Recalling  (\ref{eqlemmaghdprod}),  which had been obtained by the choice of $\xone$, we get
\begin{alignat}1\nonumber
  \int_\omega W(Dz)dx\le& c_W c_{F,\eps} c_{\eta}\delta\,.
\end{alignat}
The constant depends on $\eps$ and $F$ (and hence on $\eta$) but not
  on $\delta$ and $u$.

In $\omega_d$ instead we have $|Du-F|\circ v\le \eps$. 
The continuity estimate (\ref{eqcontwAF}) gives  then $W(Du\circ v)\le W(F)+1$
and therefore
(\ref{eqwdzfehlerterprod}) reduces to 
\begin{equation*}
  W(Dz)\le c_W (2 + W(F))(1+W(Dv)) \,.
\end{equation*}
With  (\ref{eqdefagmmaprod}) we conclude
\begin{equation*}
  \int_{\omega_d} W(Dz)dx\le  |B'| \eta \,.
\end{equation*}
The conclusion is the same.
\end{proof}
\section*{Acknowledgements}
This work was partially supported by the Deutsche Forschungsgemeinschaft
through the Forschergruppe 797  {\em ``Analysis and computation of
  microstructure in finite plasticity''}, projects CO 304/4-2 (first author)
and DO 633/2-2 (second author).


\end{document}